\documentclass{amsart}
\usepackage[hyphens]{url}
\usepackage{mathtools, microtype, mathrsfs, xfrac, hyperref, amssymb, enumerate, xspace, stmaryrd}
\usepackage[utf8]{inputenc}
\usepackage{tikz-cd}
\usepackage{pdflscape}
\usepackage{bm}

\usepackage{todonotes}

\numberwithin{equation}{section}

\numberwithin{subsection}{section}

\allowdisplaybreaks[1]

\newtheorem*{namedtheorem}{\theoremname}
\newcommand{\theoremname}{testing}

\newtheorem{theorem}{Theorem}
\newtheorem{proposition}[theorem]{Proposition}
\newtheorem{proposition-definition}[theorem]
{Proposition-Definition}
\newtheorem{corollary}[theorem]{Corollary}
\newtheorem{lemma}[theorem]{Lemma}
\newtheorem*{theorem*}{Theorem}

\theoremstyle{definition}
\newtheorem{definition}[theorem]{Definition}

\newtheorem{remark}[theorem]{Remark}

\newtheorem*{question*}{Question}

\theoremstyle{remark}

\renewcommand{\mathcal}{\mathscr}

 \newcommand\cB{\mathcal{B}}

\newcommand\cG{\mathcal{G}}

 \newcommand\cN{\mathcal{N}}
 \newcommand\cP{\mathcal{P}}

\renewcommand\AA{\mathbb{A}} 
\newcommand\CC{\mathbb{C}}

 \newcommand\NN{\mathbb{N}}
 \newcommand\PP{\mathbb{P}}

 \newcommand\bP{\mathbf{P}}

\newcommand\bW{\mathbf{W}} 
 \newcommand\bZ{\mathbf{Z}}

\newcommand\rI{\mathrm{I}}

\newcommand\arr{\ifinner\to\else\longrightarrow\fi}

\newcommand\arrto{\ifinner\mapsto\else\longmapsto\fi}

\newcommand{\eqdef}{\mathrel{\smash{\overset{\mathrm{\scriptscriptstyle def}} =}}}

\def\displaytimes_#1{\mathrel{\mathop{\times}\limits_{#1}}}

\def\displayotimes_#1{\mathrel{\mathop{\bigotimes}\limits_{#1}}}

\newcommand\aut{\operatorname{Aut}}

\newcommand\spec{\operatorname{Spec}}

\newlength{\ignora}

\newcommand{\GL}{\mathrm{GL}}
\newcommand{\SL}{\mathrm{SL}}
\newcommand{\PGL}{\mathrm{PGL}}

\newcommand{\gal}{\operatorname{Gal}}

\DeclareFontFamily{U}{mathx}{\hyphenchar\font45}
\DeclareFontShape{U}{mathx}{m}{n}{
      <5> <6> <7> <8> <9> <10>
      <10.95> <12> <14.4> <17.28> <20.74> <24.88>
      mathx10
      }{}
\DeclareSymbolFont{mathx}{U}{mathx}{m}{n}
\DeclareFontSubstitution{U}{mathx}{m}{n}
\DeclareMathAccent{\widecheck}{0}{mathx}{"71}
\DeclareMathAccent{\wideparen}{0}{mathx}{"75}

\renewcommand{\epsilon}{\varepsilon}

\begin{document}

\title{Uniform bounds for fields of definition in projective spaces}

\author{Giulio Bresciani}

\begin{abstract}
	We give a positive answer to a question of J. Doyle and J. Silverman about fields of definition of dynamical systems on $\PP^{n}$. We prove that, for fixed $n$, there exists a constant $C_{n}$ such that every dynamical system $\PP^{n}\to\PP^{n}$ is defined over an extension of degree $\le C_{n}$ of the field of moduli. More generally, the same bound works for any kind of ``algebraic structure'' defined over $\PP^{n}$, such as embedded curves, hypersurfaces, algebraic cycles. As a consequence we prove that, if $x\in X(k)$ is a rational point of an $n$-dimensional variety with quotient singularities, there exists a field extension $k'/k$ of degree $\le C_{n-1}$ such that $x$ lifts to a $k'$-rational point of any resolution of singularities.
\end{abstract}

\address{Scuola Normale Superiore\\Piazza dei Cavalieri 7\\
56126 Pisa\\ Italy}
\email{giulio.bresciani@gmail.com}

%~ \date{\arrday}

\maketitle
%~ \section{Introduction}

We work over a field $k$ of characteristic $0$ with algebraic closure $K$. A variety is a geometrically integral scheme of finite type over the base field.

Consider a variety $X$ over $K$ with some additional ``algebraic structure'' $\xi$, such as a divisor, an embedded curve, the ``empty structure'' (i.e. no additional structure), a dynamical system $X\to X$ and so forth (the concept of algebraic structure has a formal definition, see \S\ref{sect:struct} and \cite[\S 5]{giulio-angelo-moduli}).

A subextension $K/k'/k$ is a \emph{field of definition} if there exists $(Y,\eta)$ over $k'$ such that $(Y,\eta)_{K}\simeq (X,\xi)$. The \emph{field of moduli} $k_{(X,\xi)}$ is the subfield of $K$ fixed by the Galois elements $\sigma\in\gal(K/k)$ such that $\sigma^{*}(X,\xi)\simeq (X,\xi)$. 

Clearly, every field of definition contains the field of moduli. It is then natural to ask whether $(X,\xi)$ is defined over the field of moduli, or more generally what is the minimal degree over $k_{(X,\xi)}$ of a field of definition. This problem has been studied a lot starting with works of A. Weil, T. Matsusaka and G. Shimura, see e.g. \cite{weil, matsusaka, shimura-automorphic, shimura, debes-emsalem, silverman, huggins, marinatto, hutz-manes, doyle-silverman} (this list of references is far from being exhaustive).

In the case of a dynamical system $\xi$  of degree $d\ge 2$ on $\PP^{n}_{K}$, i.e. an equivalence class of morphisms $\PP^{n}_{K}\to\PP^{n}_{K}$ for the conjugacy action of $\PGL_{n+1}(K)$, J. Doyle and J. Silverman proved that there exists a bound $C_{n,d}$ depending only on $n$ and $d$ for the minimal degree of a field of definition over the field of moduli \cite[Theorem 1]{doyle-silverman}. They asked whether it is possible to give a bound $C_{n}$ depending only on $n$ \cite[Question 2]{doyle-silverman}, since for $n=1$ R. Hidalgo proved that $C_{1}=2$ works \cite{hidalgo}.

We prove that Doyle and Silverman's question has a positive answer for every $n$. More generally, we give a uniform bound $C_{n}$ which works not only for dynamical systems on $\PP^{n}_{K}$, but for every algebraic structure $\xi$ on $\PP^{n}_{K}$ such that $\aut(\xi)\subset\PGL_{n+1}(K)$ is finite (this hypothesis holds for dynamical systems \cite{levy}).

\begin{theorem}\label{thm:main}
	Fix $n\ge 0$ a non-negative integer. There exists a constant $C_{n}$ such that, for every algebraic structure $\xi$ on $\PP^{n}_{K}$ with field of moduli $k_{\xi}$ and such that $\aut(\PP^{n}_{K},\xi)\subset\PGL_{n+1}(K)$ is finite, there exists a finite extension $k'/k_{\xi}$ of degree at most $C_{n}$ and an algebraic structure $(\PP^{n}_{k'},\eta)$ on $\PP^{n}_{k'}$ such that $\eta_{K}\simeq\xi$.
\end{theorem}

The constant $C_{n}$ is uniform \emph{across all types of algebraic structures}: the same constant works for dynamical systems of arbitrary degree, hypersurfaces, cycles etc. 

In our joint paper with A. Vistoli \cite{giulio-angelo-moduli}, we have showed that the problem of fields of moduli versus fields of definition for varieties of dimension $\ge 2$ is tightly bound to the following question: given a variety $X$ with quotient singularities, a rational point $x\in X(k)$ and a resolution of singularities $\tilde{X}\to X$, when does $x$ lift to a rational point of $\tilde{X}$? We called the study of this question the \emph{arithmetic of quotient singularities}.

The $1$-dimensional case is trivial. In \cite{giulio-sing} we gave a complete classification in dimension $2$, while in \cite[\S 6]{giulio-angelo-moduli} we gave some results in arbitrary dimension. Here, we prove the following.

\begin{theorem}\label{thm:sing}
	Let $X$ be an $n$-dimensional variety with quotient singularities over $k$, $\tilde{X}\to X$ a resolution, and $x\in X(k)$ a rational point. There exists a finite extension $k'/k$ of degree $\le C_{n-1}$ such that $x$ lifts to a $k'$-rational point of $\tilde{X}$, where $C_{n-1}$ is the constant given in Theorem~\ref{thm:main}.
\end{theorem}

We remark that Theorem~\ref{thm:sing} is not only a consequence of Theorem~\ref{thm:main}, but it is used in its proof too: we use Theorem~\ref{thm:sing} in dimension $n$ to prove Theorem~\ref{thm:main} in dimension $n$, which in turn is used to prove Theorem~\ref{thm:sing} in dimension $n+1$, and so forth. This gives further evidence of the fact that the arithmetic of quotient singularities is a fundamental part of studying fields of definition versus fields of moduli problems.

It is possible to turn the proof of Theorem~\ref{thm:main} into an explicit bound for $C_{n}$, but this explicit bound is extremely large, see Remark~\ref{rmk:explicit}. As we have already said, R. Hidalgo proved that for $n=1$ we can choose $C_{1}=2$, and it is easy to show that this is optimal. For $C_{2}$, a careful analysis of the proof of \cite[Theorem 12]{giulio-p2} reveals that the optimal bound for $C_{2}$ is either $3$ or $4$.

\subsection*{Acknowledgements}

I would like to thank J. Silverman and A. Vistoli for very fruitful discussions.

\section{Algebraic structures}\label{sect:struct}

For the sake of synthesis, we call an \emph{algebraic structure} $\xi$ on $\PP^{n}_{K}$ what we called an \emph{algebraic object $(\PP^{n}_{K},\xi)$ in a category of structured spaces} in \cite[\S 5]{giulio-angelo-moduli}. We refer to \cite[\S 5]{giulio-angelo-moduli} for the precise definition. 

Here, suffice it to say that the concept of algebraic structure is very general and includes dynamical systems \cite{giulio-dynamical}, embedded subvarieties \cite{giulio-p2}, algebraic cycles \cite{giulio-divisor, giulio-points} and more or less any kind of additional structure $\xi$ one may attach to the ambient variety $X$ (in our case, $X=\PP^{n}_{K}$), as long as the automorphism group of the pair $(X,\xi)$ is of finite type \cite[Proposition 3.9]{giulio-angelo-moduli}. For an algebraic structure $\xi$, the field of moduli $k_{\xi}$ is defined \cite[Definition 3.11]{giulio-angelo-moduli} and it coincides with the Galois-theoretic definition of field of moduli whenever this makes sense \cite[Proposition 3.13]{giulio-angelo-moduli}.

\section{Group theoretic bounds}

We are going to need some group-theoretic bounds. We start with a simple bound about abelian groups.

\begin{lemma}\label{lem:abound}
	Let $n,d$ be integers, and let $a_{d}$ be the number of abelian groups of degree $\le d$ up to isomorphism. For every finite abelian group $A$ of rank at most $n$, there are at most $d^{n}a_{d}$ subgroups $B\subset A$ with $|A|\le d|B|$, i.e. of index at most $d$. 
\end{lemma}

\begin{proof}
	For every abelian group $H$ of degree $\le d$ there are at most $d^{n}$ homomorphisms $A\to H$. The statement follows.
\end{proof}

Next, we need the following theorem of C. Jordan.

\begin{theorem}[{C. Jordan, \cite{jordan} \cite[Theorem 36.13]{curtis-reiner}}]\label{thm:jordan}
	Fix $n>0$ a positive integer. There exists a constant $b_{n}$ such that every finite subgroup $G$ of $\GL_{n}(\CC)$ has an abelian, normal subgroup $A\subset G$ of index at most $b_{n}$.
\end{theorem}

We say that a subgroup $G\subset\PGL_{n}(K)$ is diagonalizable if its inverse image in $\GL_{n}(K)$ is diagonalizable. Equivalently, $G$ is diagonalizable if it is the image of a diagonalizable subgroup of $\GL_{n}(K)$.

As a consequence of Jordan's theorem, we prove the following.

\begin{corollary}\label{cor:bound}
	Fix $n>0$ a positive integer. There exists a constant $c_{n}$ such that every finite subgroup $G$ of $\PGL_{n}(\CC)$ has a diagonalizable, characteristic subgroup of index at most $c_{n}$.
\end{corollary}

\begin{proof}
	Let $G\subset\PGL_{n}(\CC)$ be a finite subgroup, $G'\subset\SL_{n}(\CC)$ its inverse image in $\SL_{n}(\CC)$. Let $A'\subset G'$ be the normal, abelian subgroup given by Jordan's theorem~\ref{thm:jordan}, it is diagonalizable since this is true for every finite, abelian subgroup of $\GL_{n}(\CC)$. Denote by $A\subset G$ the image of $A'$; by construction, $A\subset G$ is a normal, diagonalizable subgroup of index $d\le c_{n}$. Let us show that we can extract from $A$ a characteristic subgroup of $G$ of bounded index.
	
	Consider the set $S=\{H\subset G\mid H=\phi(A),~\phi\in\aut(G)\}$ of subgroups of $G$ of the form $\phi(A)$ for some automorphism $\phi\in \aut(G)$, we want to bound the size of $S$. 

	Given an element $H\in S$, notice that the index of $H\cap A$ in $A$ is equal to the index of $H\cap A$ in $H$, which in turn is at most $d$ since $H\cap A$ is the kernel of $H\to G/A$. Furthermore, notice that $H$ is normal as well, so that $H\cap A$ is normal in $G$.
	
	Let $S_{1}$ be the set of images of homomorphisms $A\to G/A$, and $S_{2}$ the set of subgroups of $A$ of index $\le d$ which are normal in $G$, we have a natural map $S\to S_{1}\times S_{2}$ given by $H\mapsto (H/(H\cap A), H\cap A)$. Bounding the size of $S$ reduces to bounding the sizes of $S_{1},S_{2}$ and of the fibers of $S\to S_{1}\times S_{2}$.
	
	First, $|S_{1}|$ is bounded by the number of homomorphisms $A\to G/A$, which in turn is bounded by $d^{n}$ since $A$ has rank $\le n$. Second, $|S_{2}|\le d^{n}a_{d}$ by Lemma~\ref{lem:abound}. Finally, the fibers of $S\to S_{1}\times S_{2}$ have at most $d^{n}$ elements. In fact, let $G'/A\in S_{1}$ be the image of an homomorphism $A\to G/A$ (where $G'$ is the inverse image in $G$ of the subgroup of $G/A$), and $A'\in S_{2}$ be a subgroup of $A$ of index at most $d$ which is normal in $G$. An element $H\in S$ is in the fiber of $(G'/A, A')\in S_{1}\times S_{2}$ if it maps surjectively on $G'/A$ and satisfies $H\cap A=A'$. Such a subgroup is identified by the induced section 
	\[G'/A\simeq H/A'\to G'/A'\]
	of $G'/A' \to G'/A$, and there are at most $d^{n}$ such sections because $G'/A$ is an abelian group of rank $\le n$ (it's a quotient of $A$) and the kernel $A/A'$ has degree at most $d$. It follows that $|S|\le d^{3n}a_{d}$.
	
	Define
	\[A_{c}\eqdef \bigcap_{\phi\in\aut(G)}\phi(A)=\bigcap_{H\in S}H=\ker(G\to \prod_{H\in S}G/H)\]
	as the intersection of all the subgroups in $S$. By construction, $A_{c}$ is a characteristic subgroup. Since $|S|\le d^{3n}\cdot a_{d}$, then $\prod_{H\in S}G/H$ has at most $d^{d^{3n}\cdot a_{d}}$ elements, hence the index of $A_{c}$ in $G$ is at most $d^{d^{3n}\cdot a_{d}}$. We can then choose
	\[c_{n}\eqdef b_{n}^{b_{n}^{3n}\cdot a_{b_{n}}}.\]
\end{proof}

\section{Proof of Theorems~\ref{thm:main} and \ref{thm:sing}}\label{sect:main}

%~ \begin{theorem}\label{thm:main}
	%~ Fix $n\ge 0$ a non-negative integer. There exists a constant $C_{n}$ such that, for every algebraic structure $\xi$ on $\PP^{n}_{K}$ with field of moduli $k_{\xi}$ and such that $\aut(\xi)\subset\PGL_{n+1}(K)$ is finite, there exists a finite extension $k'/k_{\xi}$ of degree at most $C_{n}$ and a model of $\xi$ on $\PP^{n}_{k'}$.
%~ \end{theorem}

We now want to prove Theorem~\ref{thm:main} and Theorem~\ref{thm:sing}. We will prove them at the same time by induction on $n$: the $(n-1)$-th step of Theorem~\ref{thm:main} is used to prove the $n$-th step of Theorem~\ref{thm:sing}, which in turn is used to prove the $n$-th step of Theorem~\ref{thm:main}, and so forth. The cases $n=0$ are both trivial, with constants $C_{-1}=C_{0}=1$. Assume $n\ge 1$.

First, let us recall a few definitions from \cite{giulio-angelo-moduli}.

\begin{definition}[{\cite[Definition 6.6]{giulio-angelo-moduli}}]
	A rational point $x\in X(k)$ of a variety $X$ is \emph{liftable} if it lifts to a rational point $\tilde{x}\in\tilde{X}(k)$ of one (and hence all, by Lang-Nishimura \cite[Theorem 4.1]{giulio-angelo-valuative}) resolution of singularities $\tilde{X}\to X$.
\end{definition}

\begin{definition}[{\cite[\S 6.4, Proposition 6.5]{giulio-angelo-moduli}}]
	An extension $k'/k$ is a \emph{splitting field} for $X$ if $X_{k'}$ has at least one $k'$-rational, liftable point.
\end{definition}

\begin{proposition}\label{prop:singbound}
	Assume that Theorem~\ref{thm:main} holds in dimension $n-1$. Then Theorem~\ref{thm:sing} holds in dimension $n$.
\end{proposition}

\begin{proof}
	The base change of the singularity $(X,x)$ to $K$ is equivalent to $(\AA^{n}/G,[0])$, where $G$ is the local fundamental group of the singularity \cite[Corollary 6.4]{giulio-angelo-moduli}. Let $\hat{X}\to X$ be the minimal stack of $X$ \cite[\S 6.1]{giulio-angelo-moduli}, the reduced fiber over $x$ is the \emph{fundamental gerbe} $\cG$ of the singularity \cite[Definition 6.2]{giulio-angelo-moduli}. The base change of $\cG$ to $K$ is the classifying stack $\cB_{K}G$, and $x$ is liftable if and only if $\cG(k)\neq\emptyset$ \cite[Proposition 6.5]{giulio-angelo-moduli}. Denote by $\cN$ the normal bundle of $\cG$ in $\hat{X}$, it is a vector bundle over $\cG$ of rank $n$ whose base change to $K$ corresponds to the representation $G\subset\GL_{n}(K)$. 
	
	Consider the inertia stack $\rI\to \cG$ of $\cG$, it is a relative group sheaf with an induced action on $\cN$. Denote by $\Delta\subset \rI$ the subgroup sheaf of elements acting by scalar multiplication on the vector bundle $\cN$, and let $\bar{G}=\cG/\Delta$ the rigidification of $\cG$ by $\Delta$ \cite[Appendix C]{abramovich-graber-vistoli}. By construction, the rigidification $\bar{\cG}$ is a gerbe over $k$ whose base change to $K$ is the classifying stack $\cB_{K}\bar{G}$, where $\bar{G}$ is the image of $G$ in $\PGL_{n}(K)$.
	
	Let $E$ be the associated variety of $x$ \cite[Definition 6.8]{giulio-angelo-moduli}, i.e. the coarse moduli space of the projectivization $\PP(\cN)$; it is equipped with a non-canonical rational map $E\dashrightarrow \cG$ \cite[Corollary 6.9]{giulio-angelo-moduli}. Consider the composition $E\dashrightarrow \cG\to\bar{\cG}$; by construction, its base change to $K$ is the rational map $\PP^{n-1}_{K}/\bar{G}\dashrightarrow \cB_{K}\bar{G}$ associated with the projection $\PP^{n-1}_{K}\to\PP^{n-1}_{K}/\bar{G}$. By \cite[Theorem 2]{giulio-fmod}, this defines an algebraic structure $\xi$ on $\PP^{n-1}_{K}$ with automorphism group $\bar{G}$, field of moduli $k$ and residual gerbe $\bar{G}$.
	
	Since we are assuming that Theorem~\ref{thm:main} holds for $n-1$, there exists a finite extension $k'/k$ of degree $\le C_{n-1}$ such that the algebraic structure $\xi$ is defined over $k'$. By \cite[Theorem 2]{giulio-fmod} again, the fact that $\xi$ descends to $k'$ implies that $\PP^{n-1}_{K}\to\PP^{n-1}_{K}/\bar{G}$ descends to a ramified covering $P\to E_{k'}$, where $P$ is the Brauer--Severi variety to which $\xi$ descends. However, Theorem~\ref{thm:main} does not only guarantee that $\xi$ descends to $k'$, it also guarantees that $\xi$ descends to $\PP^{n-1}_{k'}$. This means that we may choose $P=\PP^{n-1}_{k'}$, and hence we get that the $k'$-rational points of $E$ are dense. Since we have a rational map $E\dashrightarrow\cG$, this implies that $\cG(k')\neq\emptyset$, i.e. that $x_{k'}\in X_{k'}(k')$ is liftable.
\end{proof}

\begin{proposition}
	Assume that Theorem~\ref{thm:main} holds in dimension $\le n-1$. Then Theorem~\ref{thm:main} holds in dimension $n$.
\end{proposition}

\begin{proof}
	By Proposition~\ref{prop:singbound}, we know that Theorem~\ref{thm:sing} holds in dimension $\le n$. Up to replacing the constants $C_{1},\dots,C_{n-1}$, we may assume $C_{1}\le C_{2}\le\dots\le C_{n-1}$. Furthermore, up to replacing $k$ with $k_{\xi}$, we may assume that $k$ is the field of moduli.
	
	Let $G=\aut(\xi)\subset \PGL_{n+1}(K)$ be the automorphism group of $\xi$, by assumption it is finite. Let $\cG$ over $k$ be the residual gerbe c.f. \cite[\S 3.1]{giulio-angelo-moduli}, it is a finite gerbe classifying twisted forms of $(\PP^{n}_{K},\xi)$, and its base change to $K$ is the classifying stack $\cB_{K}G$.
	
	Let $\cP\to\cG$ the universal family c.f. \cite[\S 5.1]{giulio-angelo-moduli}, it is an $n$-dimensional projective bundle with the property that, if $(P,\eta)$ is a twisted form over $k$ of $(\PP^{n}_{K},\xi)$ associated with a morphism $s:\spec k\to\cG$, then $P$ identifies with the fibered product $\cP\times_{\cG}\spec k$. In general, $P$ will be a Brauer--Severi variety over $k$, and $P\simeq\PP^{n}_{k}$ if and only if $s:\spec k\to\cG$ lifts to $\cP$. It follows that $\xi$ descends to an algebraic structure on $\PP^{n}_{k}$ if and only if $\cP(k)\neq\emptyset$.
	
	Denote by $\bP$ the compression of $\xi$ c.f. \cite[Definition 5.3]{giulio-angelo-moduli}, it is the coarse moduli space of $\cP$ and it is a twisted form of $\PP^{n}_{K}/G$ over the field of moduli $k$. The base change to $K$ of the structural map $\cP\to\bP$ is the projection $[\PP^{n}_{K}/G]\to\PP^{n}_{K}/G$; since the action of $G$ on $\PP^{n}_{K}$ is faithful, this map is birational, and we get a birational inverse $\bP\dashrightarrow\cP$. By \cite[Theorem 4.1]{giulio-angelo-valuative}, this implies that $\cP(k)\neq\emptyset$ if and only if $\bP$ has a liftable $k$-rational point.
	 
	Because of this, to prove the statement it is enough to find liftable $k'$-points on $\bP_{k'}$, with $k'/k$ of bounded degree. 
	
	Let $A\subset G$ be the diagonalizable subgroup of index at most $c_{n+1}$ given by Corollary~\ref{cor:bound}. There are three cases: either $G$ is trivial, or $G$ is non-trivial and $A$ is trivial, or $A$ is non-trivial.
	
	{\bf Case 1: $G$ is trivial.}
	
	If $G$ is trivial, then $\bP$ is a Brauer--Severi variety over $k$ of dimension $n$, i.e. a twisted form of $\PP^{n}_{K}$. Such a variety is always split by a field extension of degree at most $n+1$ \cite[Theorem 2.4.3, Proposition 4.5.4, Theorem 5.2.1]{gille-szamuely}. Hence, choosing $C_{n}\ge n+1$ is sufficient for the case in which $G$ is trivial.
	
	{\bf Case 2: $G$ is non-trivial, $A$ is trivial.}
	
	If $G$ is non-trivial but $A$ is trivial, then $G$ has cardinality at most $c_{n+1}$. Choose $g\in G$ any non-trivial element and $g_{0}\in\GL_{n+1}(K)$ a lifting of $g$. Notice that $g^{k}=\lambda\operatorname{Id}$ for some $k\in\NN$, $\lambda\in K$ since $G$ is finite, hence $g_{0}$ is diagonalizable. The eigenspaces of $g_{0}$ only depend on $g$ (and not on the choice of $g_{0}$), hence the projective eigenspaces of $g$ in $\PP^{n}_{K}$ are well defined. Since $g$ is non-trivial, $\PP^{n}_{K}$ is not a projective eigenspace. 
	
	Choose $1 \le m \le n-1$ an integer such that $g$ has $m$-dimensional projective eigenspaces, and denote by $Z\subset\PP^{n}_{K}$ the union of all $m$-dimensional eigenspaces of the elements of $G$ of the form $\phi(g)$ for some $\phi\in\aut(G)$; the irreducible components of $Z$ are $m$-dimensional projective subspaces, and their number is at most $(n+1)\cdot |G|\le (n+1)\cdot c_{n+1}$.
	
	Clearly, $Z$ is distinguished in the sense of \cite[\S 7]{giulio-fmod}, hence $Z/G\subset \PP^{n}/G$ descends to a closed subset $\bZ\subset\bP$. Notice that $Z/G$ has at most $(n+1)\cdot c_{n+1}$ irreducible components as well. Choose $W\subset Z$ an $m$-dimensional eigenspace of $g$, and consider $\bW_{0}\subset \bZ$ its image in $\bZ$. By construction, $\bW_{0,K}$ has at most $n\cdot c_{n+1}$ irreducible components, hence there exists a field extension $k'/k$ of degree at most $(n+1)\cdot c_{n+1}$ such that the image $\bW\subset\bP_{k'}$ of $W$ in $\bP_{k'}$ is geometrically irreducible.
	
	Let $G_{W}$ be the subquotient of elements of $G$ mapping $W$ to itself modulo the elements acting trivially on $W$; by construction, $\bW_{K}\simeq W/G_{K}\subset \PP^{n}_{K}/G$.	By \cite[Theorem 2]{giulio-fmod}, $\bW$ is the compression of an algebraic structure on $W\simeq\PP^{m}$ with automorphism group $G_{W}$. By inductive hypothesis, there exists an extension $k''/k'$ of degree at most $C_{m}\le C_{n-1}$ such that $\bW_{k''}$ has a liftable $k''$-rational point $w\in \bW_{k''}(k'')$. By Proposition~\ref{prop:singbound}, there exists a finite extension $k'''/k''$ of degree at most $C_{n-1}$ such that $w_{k'''}$ is liftable \emph{as a point of $\bP_{k'''}$} (and not only as a point of $\bW_{k'''}$). Since
	\[[k''':k]=[k''':k''] \cdot [k'':k'] \cdot [k':k]\le C_{n-1}\cdot C_{n-1}\cdot (n+1)\cdot c_{n+1},\]
	it is sufficient to choose $C_{n}\ge (n+1)c_{n+1}C_{n-1}^{2}$ for the case in which $G$ is non-trivial but $A$ is trivial.
	
	{\bf Case 3: $A$ is non-trivial.}
	Since $A$ is diagonalizable, its projective eigenspaces are well defined. Furthermore, since $A$ is non-trivial, $\PP^{n}_{K}$ is not a projective eigenspace. Choose $1 \le m \le n-1$ an integer such that $A$ has $m$-dimensional projective eigenspaces, and denote by $Z\subset\PP^{n}_{K}$ the union of all $m$-dimensional eigenspaces of $A$. Since $A$ is characteristic, then $Z$ is distinguished in the sense of \cite[\S 7]{giulio-fmod}. Notice that $Z$ has at most $n+1$ irreducible components. The rest of the proof of this case is identical to the previous one, with the exception that the constant $(n+1)\cdot c_{n+1}$ is replaced by $n+1$. Hence, choosing $C_{n}\ge (n+1) \cdot C_{n-1}^{2}$ is sufficient for the case in which $A$ is non-trivial
	
	To sum up, since clearly $(n+1) \cdot c_{n+1} \cdot C_{n-1}^{2}\ge (n+1) \cdot C_{n-1}^{2}\ge n+1$, the statement holds for
	\[C_{n}\eqdef (n+1) \cdot c_{n+1} \cdot C_{n-1}^{2}.\]
\end{proof}

\begin{remark}\label{rmk:explicit}
	It is possible repeat the arguments above more explicitly and obtain an explicit upper bound for $C_{n}$, using estimates for the constants $a_{n}$ and $b_{n}$ of Lemma~\ref{lem:abound} and Theorem~\ref{thm:jordan} such as \cite{erdos-szekeres} and \cite{collins}. However, this upper bound is going to be wildly large, something similar to
	\[(n+1)!^{(n+1)!^{(n+1)!}}.\]
	We do not claim that this is actually an upper bound, we just want to give an idea of the result of such a computation. 
	
	We think that trying to formalize this upper bound is mostly pointless, for two reasons. The first reason is that the proofs above do not try in any way to be optimal, and frequently use larger but simpler bounds instead of smaller and more complex ones. More importantly, we think that any attempt to obtain a reasonable estimate should change strategy completely, rather than trying to optimize the one given above. Let us explain this.
	
	In the proof of Theorem~\ref{thm:main}, we used the constants $a_{n}$ and $b_{n}$ to take care of the most complex finite subgroups of $\PGL_{n+1}(\CC)$. However, there is a common, counterintuitive phenomenon in the study of fields of moduli: for a variety with a structure $(X,\xi)$, the more $\aut(X,\xi)$ is complex, the more likely it is that $(X,\xi)$ is defined over the field of moduli, see e.g. \cite{huggins}, \cite{giulio-divisor}, \cite{giulio-p2}. We think that any attempt to obtain a reasonable estimate should try to leverage this phenomenon in order to take care of the most complex automorphism groups (instead of simply blowing up the bounds, as we do), and obtain sharp estimates for the less complex ones, e.g. the abelian ones.
\end{remark}

\bibliographystyle{amsalpha}
\bibliography{bound}

\providecommand{\bysame}{\leavevmode\hbox to3em{\hrulefill}\thinspace}
\providecommand{\MR}{\relax\ifhmode\unskip\space\fi MR }
\providecommand{\MRhref}[2]{%
  \href{http://www.ams.org/mathscinet-getitem?mr=#1}{#2}
}
\providecommand{\href}[2]{#2}
\begin{thebibliography}{AGV08}

\bibitem[AGV08]{abramovich-graber-vistoli}
Dan Abramovich, Tom Graber, and Angelo Vistoli, \emph{Gromov-{W}itten theory of
  {D}eligne-{M}umford stacks}, Amer. J. Math. \textbf{130} (2008), no.~5,
  1337--1398.

\bibitem[Brea]{giulio-p2}
Giulio Bresciani, \emph{The field of moduli of plane curves}, arxiv:2303.01454.

\bibitem[Breb]{giulio-dynamical}
\bysame, \emph{Fields of definition of dynamical systems on $\mathbb{P}^{1}$.
  {I}mprovements on a result of {S}ilverman}, arxiv.

\bibitem[Bre23]{giulio-fmod}
\bysame, \emph{The field of moduli of a variety with a structure}, Boll. Unione
  Mat. Ital. (2023), \url{https://doi.org/10.1007/s40574--023--00399--z}.

\bibitem[Bre24a]{giulio-sing}
\bysame, \emph{The arithmetic of tame quotient singularities in dimension 2},
  Int. Math. Res. Not. IMRN (2024), no.~3, 2017--2043. \MR{4702270}

\bibitem[Bre24b]{giulio-divisor}
\bysame, \emph{The field of moduli of a divisor on a rational curve}, Journal
  of Algebra \textbf{647} (2024), 72--98.

\bibitem[Bre24c]{giulio-points}
\bysame, \emph{The field of moduli of sets of points in {$\Bbb P^2$}}, Arch.
  Math. (Basel) \textbf{122} (2024), no.~5, 513--519.

\bibitem[BV23]{giulio-angelo-valuative}
Giulio Bresciani and Angelo Vistoli, \emph{An arithmetic valuative criterion
  for proper maps of tame algebraic stacks}, Manuscripta Math. (2023),
  \url{https://doi.org/10.1007/s00229--023--01491--6}.

\bibitem[BV24]{giulio-angelo-moduli}
\bysame, \emph{Fields of moduli and the arithmetic of tame quotient
  singularities}, Compositio Mathematica \textbf{160} (2024), no.~5,
  982–1003.

\bibitem[Col07]{collins}
Michael~J. Collins, \emph{On {J}ordan's theorem for complex linear groups}, J.
  Group Theory \textbf{10} (2007), no.~4, 411--423.

\bibitem[CR62]{curtis-reiner}
Charles~W. Curtis and Irving Reiner, \emph{Representation theory of finite
  groups and associative algebras}, Pure and Applied Mathematics, vol. Vol. XI,
  Interscience Publishers (a division of John Wiley \& Sons, Inc.), New
  York-London, 1962.

\bibitem[DE99]{debes-emsalem}
Pierre D\`ebes and Michel Emsalem, \emph{On fields of moduli of curves}, J.
  Algebra \textbf{211} (1999), no.~1, 42--56.

\bibitem[DS19]{doyle-silverman}
John~R. Doyle and Joseph~H. Silverman, \emph{A uniform
  field-of-definition/field-of-moduli bound for dynamical systems on
  {$\Bbb{P}^N$}}, J. Number Theory \textbf{195} (2019), 1--22.

\bibitem[ES34]{erdos-szekeres}
P.~Erd\"os and G.~Szekeres, \emph{\"{U}ber die {A}nzahl der {A}belschen
  {G}ruppen gegebener {O}rdnung und \"uber ein verwandtes zahlentheoretisches
  {P}roblem}, Acta sci. Math. Szeged \textbf{VII} (1934), no.~11, 95--102.

\bibitem[GS17]{gille-szamuely}
Philippe Gille and Tam\'{a}s Szamuely, \emph{Central simple algebras and
  {G}alois cohomology}, second ed., Cambridge Studies in Advanced Mathematics,
  vol. 165, Cambridge University Press, Cambridge, 2017.

\bibitem[Hid14]{hidalgo}
Rub\'{e}n~A. Hidalgo, \emph{A simple remark on the field of moduli of rational
  maps}, Q. J. Math. \textbf{65} (2014), no.~2, 627--635.

\bibitem[HM14]{hutz-manes}
Benjamin Hutz and Michelle Manes, \emph{The field of definition for dynamical
  systems on {$\Bbb P^N$}}, Bull. Inst. Math. Acad. Sin. (N.S.) \textbf{9}
  (2014), no.~4, 585--601.

\bibitem[Hug07]{huggins}
Bonnie Huggins, \emph{Fields of moduli of hyperelliptic curves}, Math. Res.
  Lett. \textbf{14} (2007), no.~2, 249--262.

\bibitem[Jor78]{jordan}
C.~Jordan, \emph{M\'emoire sur les \'equations differentielles lin\'eaires \`a
  int\'egrate alg\'ebrique}, J. fur Math. \textbf{84} (1878), 89--215.

\bibitem[Lev11]{levy}
Alon Levy, \emph{The space of morphisms on projective space}, Acta Arith.
  \textbf{146} (2011), no.~1, 13--31.

\bibitem[Mar13]{marinatto}
Andrea Marinatto, \emph{The field of definition of point sets in
  {$\Bbb{P}^1$}}, J. Algebra \textbf{381} (2013), 176--199.

\bibitem[Mat58]{matsusaka}
T.~Matsusaka, \emph{Polarized varieties, fields of moduli and generalized
  {K}ummer varieties of polarized abelian varieties}, Amer. J. Math.
  \textbf{80} (1958), 45--82.

\bibitem[Shi59]{shimura-automorphic}
Goro Shimura, \emph{On the theory of automorphic functions}, Ann. of Math. (2)
  \textbf{70} (1959), 101--144.

\bibitem[Shi72]{shimura}
\bysame, \emph{On the field of rationality for an abelian variety}, Nagoya
  Math. J. \textbf{45} (1972), 167--178.

\bibitem[Sil95]{silverman}
Joseph~H. Silverman, \emph{The field of definition for dynamical systems on
  {$\bold P^1$}}, Compositio Math. \textbf{98} (1995), no.~3, 269--304.

\bibitem[Wei56]{weil}
Andr\'{e} Weil, \emph{The field of definition of a variety}, Amer. J. Math.
  \textbf{78} (1956), 509--524.

\end{thebibliography}

\end{document}